\documentclass[12pt]{article}
\usepackage[a4paper,margin=1.2in]{geometry}
\usepackage{amssymb}
\usepackage{amsmath}
\usepackage{amsfonts}
\usepackage{amsthm}
\usepackage{color}
\usepackage{float}
\usepackage[all]{xy}
\usepackage{scalerel}
\usepackage{mathabx}
\usepackage{delimset}
\usepackage{setspace}

\newcommand\scalemath[2]{\scalebox{#1}{\mbox{\ensuremath{\displaystyle #2}}}}

\newcommand{\qbinom}{\genfrac{[}{]}{0pt}{}}

\newcommand{\e}{\mathrm{e}}
\newcommand{\h}{\mathrm{h}}
\newcommand{\g}{\mathrm{g}}
\newcommand{\E}{\mathrm{E}}
\newcommand{\T}{\mathrm{T}}
\newcommand{\rr}{\mathrm{r}}
\newcommand{\s}{\mathrm{s}}

\newtheorem{theorem}{Theorem}

\newtheorem{definition}{Definition}

\newtheorem{corollary}{Corollary}

\begin{document}

\title{\textbf{A Note on the Rogers-Szeg\"{o} Polynomial $q$-Differential Operators}}
\author{Ronald Orozco L\'opez}
\newcommand{\Addresses}{{
  \bigskip
  \footnotesize

  \textit{E-mail address}, R.~Orozco: \texttt{rj.orozco@uniandes.edu.co}
  
}}

\maketitle

\begin{abstract}
In this paper, we introduce the Rogers-Szeg\"o deformed $q$-differential operators $\g_{n}(bD_{q}|u)$ based on $q$-differential operator $D_{q}$. The motivation for introducing the operators $\g_{n}(bD_{q})$ is that their limit turns out to be the $q$-exponential operator $\T(bD_{q})$ given by Chen. The deformed homogeneous Al-Salam-Carlitz polynomials $\Psi_{m}^{(q^{-n})}(ub,x|uq^{-1})$ can easily be represented by using the operators $\g_{n}(bD_{q}|u)$. Identities relating the new general Al-Salam-Carlitz polynomial, defined by Cao et al., the generalized, and homogeneous Al-Salam-Carlitz polynomials $\Phi_{m}^{(q^n)}(b,x|q)$ and basic hypergeometric series are given.
\end{abstract}
\noindent 2020 {\it Mathematics Subject Classification}:
Primary 05A30. Secondary 33D45, 33D15.

\noindent \emph{Keywords: } Hahn polynomials, $(s,t)$-exponential operator, Heine binomial operators, Rogers-type formula, Mehler-type formula, basic hypergeometric series.

\section{Introduction}

We begin with some notation and terminology for basic hypergeometric series \cite{gasper}. Let $\vert q\vert<1$ and the $q$-shifted factorial be defined by
\begin{align*}
    (a;q)_{n}&=\prod_{k=0}^{n-1}(1-q^{k}a),\\
    (a;q)_{\infty}&=\lim_{n\rightarrow\infty}(a;q)_{n}=\prod_{k=0}^{\infty}(1-aq^{k}).
\end{align*}
The multiple $q$-shifted factorial is defined by
\begin{align*}
    (a_{1},a_{2},\ldots,a_{m};q)_{n}&=(a_{1};q)_{n}(a_{2};q)_{n}\cdots(a_{m};q)_{n},\\
    (a_{1},a_{2},\ldots,a_{m};q)_{n}&=(a_{1};q)_{\infty}(a_{2};q)_{\infty}\cdots(a_{m};q)_{\infty}.
\end{align*}
The $q$-binomial coefficient is defined by
\begin{equation*}
\qbinom{n}{k}_{q}=\frac{(q;q)_{n}}{(q;q)_{k}(q;q)_{n-k}}.
\end{equation*}
The ${}_{r}\phi_{s}$ basic hypergeometric series is define by
\begin{equation*}
    {}_{r}\phi_{s}\left(
    \begin{array}{c}
         a_{1},a_{2},\ldots,a_{r} \\
         b_{1},\ldots,b_{s}
    \end{array}
    ;q,z
    \right)=\sum_{n=0}^{\infty}\frac{(a_{1},a_{2},\ldots,a_{r};q)_{n}}{(q,b_{1},b_{2},\ldots,b_{s};q)_{n}}[(-1)^{n}q^{\binom{n}{2}}]^{1+s-r}z^n,\ \vert z\vert<1.
\end{equation*}
In this paper, we will frequently use the $q$-binomial theorem:
\begin{equation}
    {}_1\phi_{0}(a;q,z)=\frac{(az;q)_{\infty}}{(z;q)_{\infty}}=\sum_{n=0}^{\infty}\frac{(a;q)_{n}}{(q;q)_{n}}z^{n}.
\end{equation}
The $q$-exponential $\e_{q}(z)$ is defined by
\begin{equation*}
    \e_{q}(z)=\sum_{n=0}^{\infty}\frac{z^n}{(q;q)_{n}}={}_1\phi_{0}\left(\begin{array}{c}
         0\\
         - 
    \end{array};q,-z\right)=\frac{1}{(z;q)_{\infty}}.
\end{equation*}
Another $q$-analogue of the classical exponential function is
\begin{equation*}
    \E_{q}(z)=\sum_{n=0}^{\infty}q^{\binom{n}{2}}\frac{z^n}{(q;q)_{n})}={}_1\phi_{1}\left(\begin{array}{c}
         0\\
         0 
    \end{array};q,-z\right)=(-z;q)_{\infty}.
\end{equation*}
The following easily verified identities will be frequently used in this paper:
\begin{align}
    (a;q)_{n}&=\frac{(a;q)_{\infty}}{(aq^n;q)_{\infty}},\label{eqn_iden1}\\
    (a;q)_{n+k}&=(a;q)_{n}(aq^n;q)_{k},\label{eqn_iden2}\\
    (a;q)_{n-k}&=\frac{(a;q)_{n}}{(a^{-1}q^{1-n};q)_{k}}(-qa^{-1})^{k}q^{\binom{k}{2}-nk}.\label{eqn_iden3}
\end{align}
The $q$-differential operator $D_{q}$ is defined by:
\begin{equation*}
    D_{q}f(x)=\frac{f(x)-f(qx)}{x}
\end{equation*}
and the Leibniz rule for $D_{q}$
\begin{equation}\label{eqn_leibniz}
    D_{q}^{n}\{f(x)g(x)\}=\sum_{k=0}^{n}q^{k(n-k)}\qbinom{n}{k}_{q}D_{q}^{k}\{f(x)\}D_{q}^{n-k}\{g(q^{k}x)\}.
\end{equation}
Chen and Liu \cite{chen1} introduced the $q$-exponential operator
\begin{equation*}
    \T(bD_{q})=\sum_{n=0}^{\infty}\frac{(bD_{q})^n}{(q;q)_{n}}.
\end{equation*}
In this paper, we need the general Al-Salam-Carlitz polynomials, \cite{cao}
\begin{equation}\label{eqn_GASC1}
    \phi_{n}^{
    \scalemath{0.75}{
    \left(
    \begin{array}{c}
         a,b,c  \\
         d,e
    \end{array}
    \right)}
    }
    (x,y|q)=\sum_{k=0}^{n}\qbinom{n}{k}_{q}\frac{(a,b,c;q)_{n}}{(d,e;q)_{n}}x^{n-k}y^k
\end{equation}
and
\begin{equation}\label{eqn_GASC2}
    \psi_{n}^{
    \scalemath{0.75}{
    \left(
    \begin{array}{c}
         a,b,c  \\
         d,e
    \end{array}
    \right)}
    }
    (x,y|q)=\sum_{k=0}^{n}\qbinom{n}{k}_{q}(-1)^{k}q^{k(k-n)}\frac{(a,b,c;q)_{n}}{(d,e;q)_{n}}x^{n-k}y^k,
\end{equation}
with generating functions
\begin{align*}
    \sum_{n=0}^{\infty}\phi_{n}^{
    \scalemath{0.75}{
    \left(
    \begin{array}{c}
         a,b,c  \\
         d,e
    \end{array}
    \right)}
    }
    (x,y|q)\frac{t^n}{(q;q)_{n}}&=\frac{1}{(xt;q)_{\infty}}{}_3\phi_{2}\left(\begin{array}{c}
         a,b,c\\
         d,e
    \end{array};q,yt\right),\\
    \sum_{n=0}^{\infty}\psi_{n}^{
    \scalemath{0.75}{
    \left(
    \begin{array}{c}
         a,b,c  \\
         d,e
    \end{array}
    \right)}
    }
    (x,y|q)\frac{t^n}{(q;q)_{n}}&=\frac{1}{(xt;q)_{\infty}}{}_3\phi_{3}\left(\begin{array}{c}
         a,b,c\\
         0,d,e
    \end{array};q,-yt\right).
\end{align*}
The generalized Al-Salam-Carlitz polynomials are defined by
\begin{equation*}
    \phi_{n}^{(a,b,c)}(x,y|q)=\sum_{k=0}^{n}\qbinom{n}{k}_{q}\frac{(a,b;q)_{k}}{(c;q)_{k}}x^ky^{n-k}
\end{equation*}
and
\begin{equation*}
    \psi_{n}^{(a,b,c)}(x,y|q)=\sum_{k=0}^{n}\qbinom{n}{k}_{q}(-1)^{k}q^{\binom{k+1}{2}-nk}\frac{(a,b;q)_{k}}{(c;q)_{k}}x^ky^{n-k}.
\end{equation*}
The homogeneous Al-Salam-Carlitz polynomials is
\begin{align*}
    \Psi_{n}^{(a)}(x,y|q)&=\sum_{k}^{n}\qbinom{n}{k}_{q}(-1)^kq^{\binom{k+1}{2}-nk}(a;q)_{k}x^{k}y^{n-k}.
\end{align*}
The Hahn polynomials are defined by
\begin{equation}\label{eqn_han}
    \Phi_{n}^{(a)}(x,y|q)=\sum_{k=0}^{n}\qbinom{n}{k}_{q}(a;q)_{k}x^ky^{n-k}.
\end{equation}
The generating function for the homogeneous Hahn polynomials is
\begin{equation*}
    \sum_{n=0}^{\infty}\Phi_{n}^{(\alpha)}(x,y|q)\frac{t^n}{(q;q)_{n}}=\frac{(\alpha xt;q)_{\infty}}{(xt,yt;q)_{\infty}}
\end{equation*}
with $\max\{\vert xt\vert,\vert yt\vert\}<1$. The classical Rogers-Szeg\"o polynomials 
\begin{equation*}
    \h_{n}(x|q)=\sum_{k=0}^{n}\qbinom{n}{k}_{q}x^k
\end{equation*}
is obtained when set $\alpha=0$ and $y=1$ in the homogeneous Hahn polynomials Eq.(\ref{eqn_han}). When $\alpha=0$ in Eq.(\ref{eqn_han}) we obtain the generalized Rogers-Szeg\"o:
\begin{equation*}
    \rr_{n}(x,y|q)=\sum_{k=0}^{n}\qbinom{n}{k}_{q}x^ky^{n-k}.
\end{equation*}

\section{Deformed Rogers-Szeg\"o polynomial operators}

\begin{definition}
Define the deformed Rogers-Szeg\"o operator of degree $n$ based on $D_{q}$ by
\begin{equation}
    \g_{n}\left(bD_{q}|u\right)=\sum_{k=0}^{n}\qbinom{n}{k}_{q}u^{\binom{n-k}{2}}b^kD_{q}^k.
\end{equation}
If $u=1$, define the $n$-th Rogers-Szeg\"o $q$-operator based on $D_{q}$ as
\begin{equation}
    \h_{n}\left(bD_{q}|q\right)=\sum_{k=0}^{n}\qbinom{n}{k}_{q}b^kD_{q}^k.
\end{equation}
If $u=q$, define the $n$-th $q$-shifted factorial operators based on $D_{q}$ as
\begin{equation}
    \s_{n}(bD_{q}|q)=q^{\binom{n}{2}}\sum_{k=0}^{n}\qbinom{n}{k}_{q}q^{\binom{k}{2}}(q^{1-n}b)^{k}D_{q}^{k}=q^{\binom{n}{2}}(-q^{1-n}bD_{q};q)_{n}.
\end{equation}
\end{definition}
Define the deformed $q$-exponential operator as
\begin{equation}
    \T(bD_{q},u)=\sum_{k=0}^{\infty}u^{\binom{k}{2}}\frac{(bD_{q})^k}{(q;q)_{k}}.
\end{equation}
Then
\begin{align*}
    \T(bD_{q},u)&=\lim_{n\rightarrow\infty}u^{-\binom{n}{2}}\g_{n}(u^{n-1}bD_{q}|u),\\
    \T(bD_{q})&=\lim_{n\rightarrow\infty}\h_{n}(bD_{q}|q),\\
    \E(bD_{q})&=\lim_{n\rightarrow\infty}q^{-\binom{n}{2}}\s_{n}(q^{n-1}bD_{q}|q)
\end{align*}
where 
\begin{equation*}
    \E(bD_{q})=\sum_{n=0}^{\infty}q^{\binom{n}{2}}\frac{(bD_{q})^n}{(q;q)_{n}}.
\end{equation*}
Then the operator $\T(bD_{q},u)$ generalizes to the operators $\T(bD_{q})$ and $\E(bD_{q})$.
\begin{definition}
We define the deformed homogeneous Al-Salam-Carlitz polynomials as
\begin{align*}
    \Psi_{n}^{(a)}(x,y|u)&=\sum_{k}^{n}\qbinom{n}{k}_{q}(-1)^ku^{\binom{k+1}{2}-nk}(a;q)_{k}x^{k}y^{n-k}
\end{align*}
and the deformed generalized Al-Salam-Carlitz polynomials as
\begin{equation*}
    \Psi_{n}^{(a,b,c)}(x,y|u)=\sum_{k=0}^{n}\qbinom{n}{k}_{q}u^{\binom{k+1}{2}-nk}\frac{(a,b;q)_{k}}{(c;q)_{k}}x^ky^{n-k}.
\end{equation*}
\end{definition}
\begin{theorem}
    \begin{equation*}
        \g_{n}\left(bD_{q}|u\right)\{x^m\}=
        \begin{cases}
            x^m,&\text{ if }n=0;\\
            u^{\binom{n}{2}}\Psi_{m}^{(q^{-n})}(ub,x|uq^{-1}),&\text{ if }n\geq1.
        \end{cases}
    \end{equation*}
Also,
\begin{equation*}
    \T(bD_{q},u)\{x^m\}=\lim_{n\rightarrow\infty}u^{-\binom{n}{2}}g_{n}(u^{n-1}bD_{q})\{x^m\}=(x\oplus_{1,u}b)_{q}^{(m)}.
\end{equation*}
\end{theorem}
\begin{proof}
If $n=0$, $\g_{0}(bD_{q}|q)$ is the identity operator. If $n\neq0$, then
    \begin{align*}
        \g_{n}\left(bD_{q}|u\right)\{x^m\}&=u^{\binom{n}{2}}\sum_{k=0}^{n}\qbinom{n}{k}_{q}u^{\binom{k+1}{2}-nk}b^kD_{q}^k\{x^m\}\\
        &=u^{\binom{n}{2}}\sum_{k=0}^{m}\qbinom{n}{k}_{q}u^{\binom{k+1}{2}-nk}b^k\frac{(q;q)_{m}}{(q;q)_{m-k}}x^{m-k}\\
        &=u^{\binom{n}{2}}\sum_{k=0}^{m}\qbinom{m}{k}_{q}u^{\binom{k+1}{2}-nk}\frac{(q;q)_{n}}{(q;q)_{n-k}}b^kx^{m-k}\\
        &=u^{\binom{n}{2}}\sum_{k=0}^{m}\qbinom{m}{k}_{q}(-1)^{k}(uq^{-1})^{\binom{k}{2}-nk}(q^{-n};q)_{k}(ub)^kx^{m-k}\\
        &=u^{\binom{n}{2}}\Psi_{m}^{(q^{-n})}(ub,x|uq^{-1}).
    \end{align*}
Finally,
\begin{align*}
    \T(bD_{q},u)\{x^m\}&=\sum_{n=0}^{\infty}u^{\binom{n}{2}}\frac{b^nD_{q}^n}{(q;q)_{n}}\{x^m\}\\
    &=\sum_{n=0}^{m}\frac{(q;q)_{m}}{(q;q)_{n}(q;q)_{m-n}}b^nx^{m-n}\\
    &=(x\ominus_{1,uq^{-1}}b)_{q}^{(m)}.
\end{align*}
The proof is completed.    
\end{proof}

\begin{theorem}
    \begin{equation}
        \g_{n}\left(bD_{q}|u\right)\left\{\frac{1}{(ax;q)_{\infty}}\right\}=\frac{u^{\binom{n}{2}}}{(ax;q)_{\infty}}(1\oplus_{1,u}u^{1-n}ab)_{q}^{(n)}.
    \end{equation}
\end{theorem}
\begin{proof}
    \begin{align*}
        \g_{n}\left(bD_{q}|u\right)\left\{\frac{1}{(ax;q)_{\infty}}\right\}&=u^{\binom{n}{2}}\sum_{k=0}^{n}\qbinom{n}{k}_{q}u^{\binom{k+1}{2}-nk}b^kD_{q}^k\left\{\frac{1}{(ax;q)_{\infty}}\right\}\\
        &=u^{\binom{n}{2}}\sum_{k=0}^{n}\qbinom{n}{k}_{q}u^{\binom{k+1}{2}-nk}b^k\frac{a^k}{(ax;q)_{\infty}}\\
        &=\frac{u^{\binom{n}{2}}}{(ax;q)_{\infty}}(1\oplus_{1,u}u^{1-n}ab)_{q}^{(n)}
    \end{align*}
as claimed.
\end{proof}

\begin{itemize}
    \item When $u=1$,
    \begin{equation*}
        \h_{n}(bD_{q}|q)\left\{\frac{1}{(ax;q)_{\infty}}\right\}=\frac{\h_{n}(ab|q)}{(ax;q)_{\infty}}.
    \end{equation*}
    \item When $u=q$,
    \begin{equation*}
        \s_{n}(bD_{q}|q)\left\{\frac{1}{(ax;q)_{\infty}}\right\}=\frac{q^{\binom{n}{2}}}{(ax;q)_{\infty}}(1\oplus_{1,q}q^{1-n}ab)_{q}^{(n)}=q^{\binom{n}{2}}\frac{(-q^{1-n}ab;q)_{n}}{(ax;q)_{\infty}}.
    \end{equation*}
\end{itemize}

\begin{theorem}
    \begin{equation*}
        \g_{n}\left(bD_{q};u\right)\{(ax;q)_{\infty}\}=(ax;q)_{\infty}\Psi_{n}^{(0,0,a)}(ab,1|u).
    \end{equation*}
\end{theorem}
\begin{proof}
    \begin{align*}
        \g_{n}\left(bD_{q};u\right)\{(ax;q)_{\infty}\}&=u^{\binom{n}{2}}\sum_{k=0}^{n}\qbinom{n}{k}_{q}u^{\binom{k+1}{2}-nk}b^kD_{q}^k\{(ax;q)_{\infty}\}\\
        &=u^{\binom{n}{2}}\sum_{k=0}^{n}\qbinom{n}{k}_{q}u^{\binom{k+1}{2}-nk}b^ka^k(aq^kx,q)_{\infty}\\
        &=u^{\binom{n}{2}}(ax;q)_{\infty}\sum_{k=0}^{n}\qbinom{n}{k}_{q}\frac{u^{\binom{k+1}{2}-nk}}{(ax;q)_{k}}(ab)^k\\
        &=u^{\binom{n}{2}}(ax;q)_{\infty}\Psi_{n}^{(0,0,a)}(ab,1|u)
    \end{align*}
as claimed.
\end{proof}

\begin{theorem}
    \begin{equation*}
        \g_{n}(aD_{q}|u)\left\{\frac{1}{(bx,cx;q)_{\infty}}\right\}=\frac{u^{\binom{n}{2}}}{(bx,cx;q)_{\infty}}\sum_{k=0}^{n}\qbinom{n}{k}_{q}u^{\binom{k}{2}}(u^{1-n}a)^{k}\Phi_{k}^{(cx)}(b,c|q).
    \end{equation*}
\end{theorem}
\begin{proof}
From Eq.(\ref{eqn_leibniz}), we get
\begin{align*}
    &\g_{n}(aD_{q}|u)\left\{\frac{1}{(bx,cx;q)_{\infty}}\right\}\\
        &\hspace{0.5cm}=u^{\binom{n}{2}}\sum_{k=0}^{n}\qbinom{n}{k}_{q}u^{\binom{k+1}{2}-nk}a^{k}D_{q}^{k}\left\{\frac{1}{(bx,cx;q)_{\infty}}\right\}\\
        &\hspace{0.5cm}=u^{\binom{n}{2}}\sum_{k=0}^{n}\qbinom{n}{k}_{q}u^{\binom{k+1}{2}-nk}a^{k}\sum_{m=0}^{k}q^{m(m-k)}\qbinom{k}{m}_{q}D_{q}^{m}\left\{\frac{1}{(bx;q)_{\infty}}\right\}D_{q}^{k-m}\left\{\frac{1}{(cxq^{m};q)_{\infty}}\right\}\\
        &\hspace{0.5cm}=u^{\binom{n}{2}}\sum_{k=0}^{n}\qbinom{n}{k}_{q}u^{\binom{k+1}{2}-nk}a^{k}\sum_{m=0}^{k}\qbinom{k}{m}_{q}\frac{b^m}{(bx;q)_{\infty}}\frac{c^{k-m}}{(cxq^{m};q)_{\infty}}\\
        &\hspace{0.5cm}=\frac{u^{\binom{n}{2}}}{(bx,cx;q)_{\infty}}\sum_{k=0}^{n}\qbinom{n}{k}_{q}u^{\binom{k}{2}}(u^{1-n}a)^{k}\sum_{m=0}^{k}\qbinom{k}{m}_{q}(cx;q)_{m}b^mc^{k-m}\\
        &\hspace{0.5cm}=\frac{u^{\binom{n}{2}}}{(bx,cx;q)_{\infty}}\sum_{k=0}^{n}\qbinom{n}{k}_{q}u^{\binom{k}{2}}(u^{1-n}a)^{k}\Phi_{k}^{(cx)}(b,c|q)
\end{align*}
as claimed.
\end{proof}

\begin{theorem}
    \begin{equation*}
        \sum_{k=0}^{\infty}\Phi_{k}^{(cx)}(b,c|q)\frac{a^{k}}{(q;q)_{k}}=\frac{(abcx;q)_{\infty}}{(ac,ab;q)_{\infty}}.
    \end{equation*}
\end{theorem}
\begin{proof}
On the one side, from the above theorem with $u=1$
\begin{align*}
    \T(aD_{q})\left\{\frac{1}{(bc,cx;q)_{\infty}}\right\}&=\lim_{n\rightarrow\infty}\frac{1}{(bx,cx;q)_{\infty}}\sum_{k=0}^{n}\qbinom{n}{k}_{q}a^{k}\Phi_{k}^{(cx)}(b,c|q)\\
    &=\frac{1}{(bx,cx;q)_{\infty}}\sum_{k=0}^{\infty}\Phi_{k}^{(cx)}(b,c|q)\frac{a^{k}}{(q;q)_{k}}.
\end{align*}
On the other hand, from \cite{chen2}
\begin{align*}
    \T(aD_{q})\left\{\frac{1}{(bx,cx;q)_{\infty}}\right\}&=\frac{(a b cx;q)_{\infty}}{(bx,cx,ac,ab;q)_{\infty}}.
\end{align*}
\end{proof}

\begin{theorem}
    \begin{equation*}
        \E(aD_{q})\left\{\frac{1}{(bx,cx;q)_{\infty}}\right\}=\frac{(ac;q)_{\infty}}{(bx,cx;q)_{\infty}}{}_2\phi_{1}\left(\begin{array}{c}
         cx,0\\
         ac
    \end{array};q,ab\right).
    \end{equation*}
\end{theorem}
\begin{proof}
On the one side, from the above theorem with $u=q$
\begin{align*}
    \E(aD_{q})\left\{\frac{1}{(bx,cx;q)_{\infty}}\right\}&=\lim_{n\rightarrow\infty}u^{-\binom{n}{2}}\s_{n}(q^{n-1}aD_{q}|q)\\
    &=\lim_{n\rightarrow\infty}\frac{1}{(bx,cx;q)_{\infty}}\sum_{k=0}^{n}\qbinom{n}{k}_{q}q^{\binom{k}{2}}a^{k}\Phi_{k}^{(cx)}(b,c|q)\\
    &=\frac{1}{(bx,cx;q)_{\infty}}\sum_{k=0}^{\infty}q^{\binom{k}{2}}\Phi_{k}^{(cx)}(b,c|q)\frac{a^{k}}{(q;q)_{k}}\\
    &=\frac{1}{(bx,cx;q)_{\infty}}\sum_{k=0}^{\infty}q^{\binom{k}{2}}\T(cx,b;D_{q})\{c^{k}\}\frac{a^{k}}{(q;q)_{k}}\\
    &=\frac{1}{(bx,cx;q)_{\infty}}\T(cx,b;D_{q})\left\{\sum_{k=0}^{\infty}q^{\binom{k}{2}}\frac{(ac)^{k}}{(q;q)_{k}}\right\}\\
    &=\frac{1}{(bx,cx;q)_{\infty}}\T(cx,b;D_{q})\left\{(ac;q)_{\infty}\right\}\\
    &=\frac{1}{(bx,cx;q)_{\infty}}\sum_{k=0}^{\infty}\frac{(cx;q)_{k}}{(q;q)_{k}}b^kD_{q}^k\left\{(ac;q)_{\infty}\right\}\\
    &=\frac{1}{(bx,cx;q)_{\infty}}\sum_{k=0}^{\infty}\frac{(cx;q)_{k}}{(q;q)_{k}}b^ka^k(acq^k;q)_{\infty}\\
    &=\frac{(ac;q)_{\infty}}{(bx,cx;q)_{\infty}}\sum_{k=0}^{\infty}\frac{(cx;q)_{k}}{(q;q)_{k}(ac;q)_{k}}(ab)^k\\
    &=\frac{(ac;q)_{\infty}}{(bx,cx;q)_{\infty}}{}_2\phi_{1}\left(\begin{array}{c}
         cx,0\\
         ac
    \end{array};q,ab\right)
\end{align*}
as claimed.
\end{proof}

\begin{theorem}\label{theo_abc}
    \begin{multline*}
        \g_{n}(dD_{q}|u)\left\{\frac{(cx;q)_{\infty}}{(ax,bx;q)_{\infty}}\right\}\\
        =\frac{(cx;q)_{\infty}}{(ax,bx;q)_{\infty}}u^{\binom{n}{2}}\sum_{l=0}^{n}\sum_{m=0}^{l}\qbinom{n}{l}\qbinom{l}{m}_{q}u^{\binom{l}{2}}\frac{\Phi_{m}^{(bx)}(a,b)}{(cx;q)_{l}}(u^{1-n}d)^lc^{l-m}.
    \end{multline*}
\end{theorem}
\begin{proof}
From Eq.(\ref{eqn_leibniz}), we get
    \begin{align*}
        &\g_{n}(dD_{q}|u)\left\{\frac{(cx;q)_{\infty}}{(ax,bx;q)_{\infty}}\right\}\\
        &\hspace{0.5cm}=u^{\binom{n}{2}}\sum_{l=0}^{n}\qbinom{n}{l}u^{\binom{l+1}{2}-nl}d^lD_{q}^{l}\left\{\frac{(cx;q)_{\infty}}{(ax,bx;q)_{\infty}}\right\}\\
        &\hspace{0.5cm}=u^{\binom{n}{2}}\sum_{l=0}^{n}\qbinom{n}{l}u^{\binom{l}{2}}(u^{1-n}d)^l\sum_{m=0}^{l}\qbinom{l}{m}_{q}q^{m(m-l)}D_{q}^{m}\left\{\frac{1}{(ax,bx;q)_{\infty}}\right\}D_{q}^{l-m}\{(q^{m}cx;q)_{\infty}\}.
    \end{align*}
As
\begin{align*}
    D_{q}^{m}\left\{\frac{1}{(ax,bx;q)_{\infty}}\right\}&=\sum_{k=0}^{m}q^{k(k-m)}\qbinom{m}{k}_{q}D_{q}^{k}\left\{\frac{1}{(ax;q)_{\infty}}\right\}D_{q}^{m-k}\left\{\frac{1}{(bxq^{k};q)_{\infty}}\right\}\\
    &=\sum_{k=0}^{m}q^{k(k-m)}\qbinom{m}{k}_{q}\frac{a^k}{(ax;q)_{\infty}}\frac{(bq^k)^{m-k}}{(bxq^{k};q)_{\infty}}\\
    &=\sum_{k=0}^{m}\qbinom{m}{k}_{q}\frac{a^k}{(ax;q)_{\infty}}\frac{b^{m-k}}{(bxq^{k};q)_{\infty}}\\
    &=\frac{1}{(ax,bx;q)_{\infty}}\sum_{k=0}^{m}\qbinom{m}{k}_{q}(bx;q)_{k}a^kb^{m-k}\\
    &=\frac{\Phi_{m}^{(bx)}(a,b)}{(ax,bx;q)_{\infty}},
\end{align*}
then
\begin{align*}
    &\g_{n}(dD_{q}|u)\left\{\frac{(cx;q)_{\infty}}{(ax,bx;q)_{\infty}}\right\}\\
    &\hspace{1cm}=\frac{1}{(ax,bx;q)_{\infty}}u^{\binom{n}{2}}\sum_{l=0}^{n}\qbinom{n}{l}u^{\binom{l}{2}}(q^{l}cx;q)_{\infty}(u^{1-n}d)^l\sum_{m=0}^{l}\qbinom{l}{m}_{q}\Phi_{m}^{(bx)}(a,b)c^{l-m}\\
    &\hspace{1cm}=\frac{(cx;q)_{\infty}}{(ax,bx;q)_{\infty}}u^{\binom{n}{2}}\sum_{l=0}^{n}\qbinom{n}{l}u^{\binom{l}{2}}\frac{(u^{1-n}d)^l}{(cx;q)_{l}}\sum_{m=0}^{l}\qbinom{l}{m}_{q}\Phi_{m}^{(bx)}(a,b)c^{l-m}\\
    &\hspace{1cm}=\frac{(cx;q)_{\infty}}{(ax,bx;q)_{\infty}}u^{\binom{n}{2}}\sum_{l=0}^{n}\sum_{m=0}^{l}\qbinom{n}{l}\qbinom{l}{m}_{q}u^{\binom{l}{2}}\frac{\Phi_{m}^{(bx)}(a,b)}{(cx;q)_{l}}(u^{1-n}d)^lc^{l-m}.
\end{align*}
The proof is reached.
\end{proof}

\begin{theorem}
    \begin{multline*}
        \E(dD_{q})\left\{\frac{(cx;q)_{\infty}}{(ax,bx;q)_{\infty}}\right\}\\
        =\frac{(cx;q)_{\infty}}{(ax,bx;q)_{\infty}}\lim_{n\rightarrow\infty}\sum_{m=0}^{\infty}q^{\binom{m}{2}}\frac{\Phi_{m}^{(bx)}(a,b)}{(q;q)_{m}(cx;q)_{m}}d^m{}_1\phi_{1}\left(\begin{array}{c}
         0\\
         cq^mx
    \end{array};q,-q^mcd\right).
    \end{multline*}
\end{theorem}
\begin{proof}
From the above theorem with $u=q$,
\begin{align*}
    &\E(dD_{q})\left\{\frac{(cx;q)_{\infty}}{(ax,bx;q)_{\infty}}\right\}\\
    &\hspace{1cm}=\frac{(cx;q)_{\infty}}{(ax,bx;q)_{\infty}}\lim_{n\rightarrow\infty}\sum_{l=0}^{n}\sum_{m=0}^{l}\qbinom{n}{l}\qbinom{l}{m}_{q}q^{\binom{l}{2}}\frac{\Phi_{m}^{(bx)}(a,b)}{(cx;q)_{l}}d^lc^{l-m}\\
    &\hspace{1cm}=\frac{(cx;q)_{\infty}}{(ax,bx;q)_{\infty}}\lim_{n\rightarrow\infty}\sum_{l=0}^{\infty}\sum_{m=0}^{l}\frac{1}{(q;q)_{m}(q;q)_{l-m}}q^{\binom{l}{2}}\frac{\Phi_{m}^{(bx)}(a,b)}{(cx;q)_{l}}d^lc^{l-m}\\
    &\hspace{1cm}=\frac{(cx;q)_{\infty}}{(ax,bx;q)_{\infty}}\lim_{n\rightarrow\infty}\sum_{m=0}^{\infty}q^{\binom{m}{2}}\frac{\Phi_{m}^{(bx)}(a,b)}{(q;q)_{m}(cx;q)_{m}}d^m\sum_{l=0}^{\infty}\frac{q^{\binom{l}{2}}(dc)^{l}}{(q;q)_{l}(cq^mx;q)_{l}}\\
    &\hspace{1cm}=\frac{(cx;q)_{\infty}}{(ax,bx;q)_{\infty}}\lim_{n\rightarrow\infty}\sum_{m=0}^{\infty}q^{\binom{m}{2}}\frac{\Phi_{m}^{(bx)}(a,b)}{(q;q)_{m}(cx;q)_{m}}d^m{}_1\phi_{1}\left(\begin{array}{c}
         0\\
         cq^mx
    \end{array};q,-q^mcd\right).
\end{align*}
\end{proof}

\begin{theorem}
    \begin{multline*}
        \g_{n}(dD_{q}|u)\left\{\frac{1}{(ax,bx,cx;q)_{\infty}}\right\}\\     =\frac{1}{(ax,bx,cx;q)_{\infty}}u^{\binom{n}{2}}\sum_{k=0}^{n}\sum_{l=0}^{k}\qbinom{n}{k}_{q}\qbinom{k}{l}_{q}u^{\binom{k}{2}-nk}d^k\Phi_{l}^{(bx)}(a,b)(cx;q)_{l}c^{k-l}.
    \end{multline*}
\end{theorem}
\begin{proof}
From Eq.(\ref{eqn_leibniz}), we get
\begin{align*}
    &\g_{n}(dD_{q}|u)\left\{\frac{1}{(ax,bx,cx)_{\infty}}\right\}\\
    &\hspace{1cm}=u^{\binom{n}{2}}\sum_{k=0}^{n}\qbinom{n}{k}_{q}u^{\binom{k}{2}-nk}d^kD_{q}\left\{\frac{1}{(ax,bx,cx;q)_{\infty}}\right\}\\
    &\hspace{1cm}=u^{\binom{n}{2}}\sum_{k=0}^{n}\qbinom{n}{k}_{q}u^{\binom{k}{2}-nk}d^k\sum_{l=0}^{k}\qbinom{k}{l}_{q}q^{l(l-k)}D_{q}^{l}\left\{\frac{1}{(ax,bx;q)_{\infty}}\right\}D_{q}^{k-l}\left\{\frac{1}{(cq^lx;q)_{\infty}}\right\}\\
    &\hspace{1cm}=u^{\binom{n}{2}}\sum_{k=0}^{n}\qbinom{n}{k}_{q}u^{\binom{k}{2}-nk}d^k\sum_{l=0}^{k}\qbinom{k}{l}_{q}\frac{\Phi_{l}^{(bx)}(a,b)}{(ax,bx;q)_{\infty}}\frac{c^{k-l}}{(cq^lx;q)_{\infty}}\\
    &\hspace{1cm}=\frac{1}{(ax,bx,cx;q)_{\infty}}u^{\binom{n}{2}}\sum_{k=0}^{n}\qbinom{n}{k}_{q}u^{\binom{k}{2}-nk}d^k\sum_{l=0}^{k}\qbinom{k}{l}_{q}\Phi_{l}^{(bx)}(a,b)(cx;q)_{l}c^{k-l}.
\end{align*}
The proof is reached.
\end{proof}

\begin{theorem}
    \begin{multline*}
        \E(dD_{q})\left\{\frac{1}{(ax,bx,cx;q)_{\infty}}\right\}\\     =\frac{(cd;q)_{\infty}}{(ax,bx,cx;q)_{\infty}}\sum_{k=0}^{\infty}q^{\binom{k}{2}}\frac{(bx;q)_{k}(cx;q)_{k}}{(q;q)_{k}(cd;q)_{k}}(ad)^k{}_1\phi_{1}\left(\begin{array}{c}
         q^kcx\\
         q^kcd
    \end{array};q,-q^kbd\right).
    \end{multline*}
\end{theorem}
\begin{proof}
    \begin{align*}
        &\E(dD_{q})\left\{\frac{1}{(ax,bx,cx;q)_{\infty}}\right\}\\     &\hspace{1cm}=\frac{1}{(ax,bx,cx;q)_{\infty}}\lim_{n\rightarrow\infty}\sum_{k=0}^{n}\sum_{l=0}^{k}\qbinom{n}{k}_{q}\qbinom{k}{l}_{q}q^{\binom{k}{2}}d^k\Phi_{l}^{(bx)}(a,b)(cx;q)_{l}c^{k-l}\\
        &\hspace{1cm}=\frac{1}{(ax,bx,cx;q)_{\infty}}\sum_{k=0}^{\infty}\sum_{l=0}^{k}\frac{1}{(q;q)_{l}(q;q)_{k-l}}q^{\binom{k}{2}}d^k\Phi_{l}^{(bx)}(a,b)(cx;q)_{l}c^{k-l}\\
        &\hspace{1cm}=\frac{1}{(ax,bx,cx;q)_{\infty}}\sum_{l=0}^{\infty}\frac{\Phi_{l}^{(bx)}(a,b)(cx;q)_{l}}{(q;q)_{l}}\sum_{k=l}^{\infty}\frac{1}{(q;q)_{k-l}}q^{\binom{k}{2}}d^kc^{k-l}\\
        &\hspace{1cm}=\frac{1}{(ax,bx,cx;q)_{\infty}}\sum_{l=0}^{\infty}q^{\binom{l}{2}}\frac{\Phi_{l}^{(bx)}(a,b)(cx;q)_{l}}{(q;q)_{l}}d^l\sum_{k=0}^{\infty}\frac{1}{(q;q)_{k}}q^{\binom{k}{2}}(q^ldc)^{k}\\
        &\hspace{1cm}=\frac{1}{(ax,bx,cx;q)_{\infty}}\sum_{l=0}^{\infty}q^{\binom{l}{2}}\frac{\Phi_{l}^{(bx)}(a,b)(cx;q)_{l}}{(q;q)_{l}}d^l(q^lcd;q)_{\infty}\\
        &\hspace{1cm}=\frac{(cd;q)_{\infty}}{(ax,bx,cx;q)_{\infty}}\sum_{l=0}^{\infty}q^{\binom{l}{2}}\frac{(cx;q)_{l}}{(q;q)_{l}(cd;q)_{l}}d^l\sum_{k=0}^{l}\qbinom{l}{k}_{q}(bx;q)_{k}a^kb^{l-k}\\
        &\hspace{1cm}=\frac{(cd;q)_{\infty}}{(ax,bx,cx;q)_{\infty}}\sum_{k=0}^{\infty}q^{\binom{k}{2}}\frac{(bx;q)_{k}(cx;q)_{k}}{(q;q)_{k}(cd;q)_{k}}(ad)^k\sum_{l=0}^{\infty}q^{\binom{l}{2}}
        \frac{(q^kcx;q)_{l}}{(q;q)_{l}(q^kcd;q)_{l}}(q^kbd)^{l}\\
        &\hspace{1cm}=\frac{(cd;q)_{\infty}}{(ax,bx,cx;q)_{\infty}}\sum_{k=0}^{\infty}q^{\binom{k}{2}}\frac{(bx;q)_{k}(cx;q)_{k}}{(q;q)_{k}(cd;q)_{k}}(ad)^k{}_1\phi_{1}\left(\begin{array}{c}
         q^kcx\\
         q^kcd
    \end{array};q,-q^kbd\right)
    \end{align*}
as claimed.    
\end{proof}

\section{Deformed generating function}

\begin{theorem}
    \begin{equation*}
        \sum_{m=0}^{\infty}\Psi_{m}^{(q^{-n})}(b,x|uq^{-1})y^m=\frac{1}{(1-xy)(q;q)_{n}}\Psi_{n}^{(0,q,qxy)}(by,1|u).
    \end{equation*}
\end{theorem}
\begin{proof}
We have
    \begin{align*}
        \sum_{m=0}^{\infty}\Psi_{m}^{(q^{-n})}(b,x|uq^{-1})y^m&=u^{-\binom{n}{2}}\sum_{m=0}^{\infty}g_{n}(bD_{q}|u)\{x^m\}y^m\\
        &=u^{-\binom{n}{2}}\g_{n}(bD_{q}|u)\left\{\sum_{m=0}^{\infty}(xy)^m\right\}\\
        &=\sum_{k=0}^{n}\qbinom{n}{k}_{q}u^{\binom{k+1}{2}-nk}b^kD_{q}^k\left\{\frac{1}{1-xy}\right\}\\
        &=\frac{1}{(1-xy)(q;q)_{n}}\sum_{k=0}^{n}\qbinom{n}{k}_{q}u^{\binom{k+1}{2}-nk}\frac{(q;q)_{k}}{(qxy;q)_{k}}(by)^k\\
        &=\frac{1}{(1-xy)(q;q)_{n}}\Psi_{n}^{(0,q,qxy)}(by,1|u)
    \end{align*}
as claimed.
\end{proof}

\begin{corollary}
    \begin{align*}
        \sum_{m=0}^{\infty}\Psi_{m}^{(q^{-n})}(b,x|qq^{-1})y^m
        &=\frac{1}{(1-xy)(q;q)_{n}}\psi_{n}^{(0,q,qxy)}(-by,1|q).\\
        \sum_{m=0}^{\infty}\Psi_{m}^{(q^{-n})}(b,x|q^{-1})y^m
        &=\frac{1}{(1-xy)(q;q)_{n}}\phi_{n}^{(0,q,qxy)}(by,1|q).
    \end{align*}
\end{corollary}

\begin{theorem}
    \begin{equation*}
        \sum_{m=0}^{\infty}q^{-\binom{m}{2}}\Psi_{m}^{(q^{-n})}(b,x|uq^{-1})y^m=\sum_{k=0}^{n}\qbinom{n}{k}_{q}u^{\binom{k+1}{2}-nk}q^{-\binom{k}{2}}(q;q)_{k}(1\ominus_{1,1}qx)_{q}^{(-k-1)}(by)^k.
    \end{equation*}
\end{theorem}
\begin{proof}
We have
    \begin{align*}
        \sum_{m=0}^{\infty}q^{-\binom{m}{2}}\Psi_{m}^{(q^{-n})}(b,x|uq^{-1})y^m&=u^{-\binom{n}{2}}\g_{n}(bD_{q}|uq^{-1})\left\{\sum_{m=0}^{\infty}q^{-\binom{m}{2}}(xy)^m\right\}\\
        &=u^{-\binom{n}{2}}\g_{n}(bD_{q}|uq^{-1})\left\{\Theta_{0}(xy,q^{-1})\right\}\\
        &=\sum_{k=0}^{n}\qbinom{n}{k}_{q}u^{\binom{k+1}{2}-nk}b^kD_{q}^k\{\Theta_{0}(xy,q^{-1})\}\\
        &=\sum_{k=0}^{n}\qbinom{n}{k}_{q}u^{\binom{k+1}{2}-nk}q^{-\binom{k}{2}}(q;q)_{k}(1\ominus_{1,1}qx)_{q}^{(-k-1)}(by)^k
    \end{align*}
as claimed.    
\end{proof}

\begin{theorem}
    \begin{equation*}
        \sum_{m=0}^{\infty}q^{\binom{m}{2}}\Psi_{m}^{(q^{-n})}(b,x|uq^{-1})y^m=\sum_{k=0}^{n}\qbinom{n}{k}_{q}u^{\binom{k+1}{2}-nk}q^{\binom{k}{2}}(q;q)_{k}(1\ominus_{1,q^2}q^{2k+1}x)_{q}^{(-k-1)}(by)^k.
    \end{equation*}
\end{theorem}
\begin{proof}
We have
    \begin{align*}
        \sum_{m=0}^{\infty}q^{\binom{m}{2}}\Psi_{m}^{(q^{-n})}(b,x|uq^{-1})y^m&=u^{-\binom{n}{2}}\g_{n}(bD_{q}|uq^{-1})\left\{\sum_{m=0}^{\infty}q^{\binom{m}{2}}(xy)^m\right\}\\
        &=u^{-\binom{n}{2}}\g_{n}(bD_{q}|uq^{-1})\left\{\Theta_{0}(xy,q)\right\}\\
        &=\sum_{k=0}^{n}\qbinom{n}{k}_{q}u^{\binom{k+1}{2}-nk}b^kD_{q}^k\{\Theta_{0}(xy,q)\}\\
        &=\sum_{k=0}^{n}\qbinom{n}{k}_{q}u^{\binom{k+1}{2}-nk}q^{\binom{k}{2}}(q;q)_{k}(1\ominus_{1,q^2}q^{2k+1}x)_{q}^{(-k-1)}(by)^k
    \end{align*}
as claimed.
\end{proof}

\begin{theorem}
    \begin{equation*}
        \sum_{m=0}^{\infty}\Psi_{m}^{(q^{-n})}(b,x|uq^{-1})\frac{y^m}{(q;q)_{m}}=\frac{(1\oplus_{1,u}u^{1-n}by)_{q}^{(n)}}{(xy;q)_{\infty}}.
    \end{equation*}
\end{theorem}
\begin{proof}
We have
    \begin{align*}
        \sum_{m=0}^{\infty}\Psi_{m}^{(q^{-n})}(b,x|uq^{-1})\frac{y^m}{(q;q)_{m}}&=u^{-\binom{n}{2}}\g_{n}(bD_{q}|u)\left\{\sum_{m=0}^{\infty}\frac{(xy)^m}{(q;q)_{m}}\right\}\\
        &=u^{-\binom{n}{2}}\g_{n}(bD_{q}|u)\left\{\frac{1}{(xy;q)_{\infty}}\right\}\\
        &=\frac{(1\oplus_{1,u}u^{1-n}by)_{q}^{(n)}}{(xy;q)_{\infty}}.
    \end{align*}
\end{proof}

\begin{theorem}
    \begin{equation*}
        \sum_{m=0}^{\infty}q^{\binom{m}{2}}\Psi_{m}^{(q^{-n})}(b,x|uq^{-1})\frac{y^m}{(q;q)_{m}}=(xy;q)_{\infty}\Psi_{n}^{(0,0,y)}(by,1|u).
    \end{equation*}
\end{theorem}
\begin{proof}
    \begin{align*}
        \sum_{m=0}^{\infty}q^{\binom{m}{2}}\Psi_{m}^{(q^{-n})}(b,x|uq^{-1})\frac{y^m}{(q;q)_{m}}&=u^{-\binom{n}{2}}\g_{n}(bD_{q}|u)\left\{\sum_{m=0}^{\infty}\frac{(xy)^m}{(q;q)_{m}}\right\}\\
        &=u^{-\binom{n}{2}}\g_{n}(bD_{q}|u)\left\{\sum_{m=0}^{\infty}q^{\binom{m}{2}}\frac{(xy)^m}{(q;q)_{m}}\right\}\\
        &=(xy;q)_{\infty}\Psi_{n}^{(0,0,y)}(by,1|u)
    \end{align*}
\end{proof}

\section{Mehler and Rogers type formulas}

\begin{theorem}[{\bf Ordinary Mehler's formula}]
    \begin{align*}
        &\sum_{k=0}^{\infty}\Psi_{k}^{(q^{-n})}(b,x|uq^{-1})\Psi_{k}^{(q^{-m})}(c,y|vq^{-1})z^k\\
        &\hspace{2cm}=\frac{1}{(q;q)_{m}}\sum_{k=0}^{n}\qbinom{n}{k}_{q}u^{\binom{k+1}{2}-nk}b^k\sum_{l=0}^{m}\qbinom{m}{k}_{q}\frac{v^{\binom{l+1}{2}-ml}(q;q)_{l}(cz)^l(q;q)_{l}}{(1-q^{l}xyz)(xyz;q)_{l}(q;q)_{l-k}}x^{l-k}\\
    &\hspace{4cm}\times\psi_{k}^{
    \scalemath{0.75}{
    \left(
    \begin{array}{c}
         0,0,q^{l+1}\\
          q^{l+1}xyz,q^{l-k+1}
    \end{array}
    \right)}}(-q^lxyz,1|q).
    \end{align*}
\end{theorem}
\begin{proof}
    \begin{align*}
        &\sum_{k=0}^{\infty}\Psi_{k}^{(q^{-n})}(b,x|uq^{-1})\Psi_{k}^{(q^{-m})}(c,y|vq^{-1})z^k\\
        &\hspace{0cm}=u^{-\binom{n}{2}}\sum_{k=0}^{\infty}\g_{n}(bD_{q}|u)\{x^k\}\Psi_{k}^{(q^{-m})}(c,y|vq^{-1})z^k\\
        &\hspace{0cm}=u^{-\binom{n}{2}}\g_{n}(bD_{q}|u)\left\{\sum_{k=0}^{\infty}\Psi_{k}^{(q^{-m})}(c,y|vq^{-1})(xz)^k\right\}\\
        &\hspace{0cm}=u^{-\binom{n}{2}}\g_{n}(bD_{q}|u)\left\{\frac{1}{(1-xyz)(q;q)_{m}}\Psi_{m}^{(0,q,qxyz)}(cxz,1|v)\right\}\\
        &\hspace{0cm}=\frac{1}{(q;q)_{m}}\sum_{k=0}^{n}\qbinom{n}{k}_{q}u^{\binom{k+1}{2}-nk}b^kD_{q}^k\left\{\sum_{l=0}^{m}\qbinom{m}{k}_{q}v^{\binom{l+1}{2}-ml}\frac{(q;q)_{l}}{(xyz;q)_{l+1}}(cxz)^l\right\}.
    \end{align*}
As
\begin{align*}
    &D_{q}^{k}\left\{\frac{(cxz)^l}{(xyz;q)_{l+1}}\right\}\\
    &=\sum_{i=0}^{k}\qbinom{k}{i}_{q}q^{i(i-k)}D_{q}^{i}\left\{\frac{1}{(xyz;q)_{l+1}}\right\}D_{q}^{k-i}(q^{i}cxz)^l\\
    &=\sum_{i=0}^{k}\qbinom{k}{i}_{q}q^{i(i-k)}\frac{(yz)^i(q^{l+1};q)_{i}}{(xyz;q)_{l}(q^lxyz;q)_{i+1}}\frac{(q^icz)}{}\\
    &=\sum_{i=0}^{k}\qbinom{k}{i}_{q}q^{i(i-k)}\frac{(yz)^i(q^{l+1};q)_{i}}{(xyz;q)_{l}(q^lxyz;q)_{i+1}}\frac{(q^icz)^l(q;q)_{l}}{(q;q)_{l-k}(q^{l-k+i};q)_{i}}x^{l-k+i}\\
    &=\frac{(cz)^l(q;q)_{l}}{(1-q^{l}xyz)(xyz;q)_{l}(q;q)_{l-k}}x^{l-k}\sum_{i=0}^{k}\qbinom{k}{i}_{q}\frac{q^{i(i-k)}(q^{l+1};q)_{i}}{(q^{l+1}xyz;q)_{i}(q^{l-k+1};q)_{i}}(q^lyzx)^{i}\\
    &=\frac{(cz)^l(q;q)_{l}}{(1-q^{l}xyz)(xyz;q)_{l}(q;q)_{l-k}}x^{l-k}\psi_{k}^{
    \scalemath{0.75}{
    \left(
    \begin{array}{c}
         0,0,q^{l+1}\\
          q^{l+1}xyz,q^{l-k+1}
    \end{array}
    \right)}}(-q^lxyz,1|q),
\end{align*}
then
\begin{align*}
    &\sum_{k=0}^{\infty}\Psi_{k}^{(q^{-n})}(b,x|uq^{-1})\Psi_{k}^{(q^{-m})}(c,y|vq^{-1})z^k\\
    &\hspace{1cm}=\frac{1}{(q;q)_{m}}\sum_{k=0}^{n}\qbinom{n}{k}_{q}u^{\binom{k+1}{2}-nk}b^k\sum_{l=0}^{m}\qbinom{m}{k}_{q}\frac{v^{\binom{l+1}{2}-ml}(q;q)_{l}(cz)^l(q;q)_{l}}{(1-q^{l}xyz)(xyz;q)_{l}(q;q)_{l-k}}x^{l-k}\\
    &\hspace{4cm}\times\psi_{k}^{
    \scalemath{0.75}{
    \left(
    \begin{array}{c}
         0,0,q^{l+1}\\
          q^{l+1}xyz,q^{l-k+1}
    \end{array}
    \right)}}(-q^lxyz,1|q)
\end{align*}
as claimed.
\end{proof}

\begin{theorem}[{\bf Mehler's formula}]
    \begin{align*}
        &\sum_{k=0}^{\infty}\Psi_{k}^{(q^{-n})}(b,x|uq^{-1})\Psi_{k}^{(q^{-m})}(c,y|vq^{-1})\frac{z^k}{(q;q)_{k}}\\
        &\hspace{2cm}=\frac{1}{(xyz;q)_{\infty}}\sum_{k=0}^{n}\qbinom{n}{k}_{q}u^{\binom{k+1}{2}-nk}b^k\\
        &\hspace{1cm}\times\sum_{l=0}^{k}\qbinom{k}{l}_{q}\qbinom{m}{k-l}_{q}v^{\binom{k-l}{2}}(q;q)_{k-l}(1\oplus_{1,v}v^{1-m+k-l}q^{l}cxz)_{q}^{(m-k+l)}(yz)^l(v^{1-m}q^lcz)^{k-l}.
    \end{align*}
\end{theorem}
\begin{proof}
    \begin{align*}
        &\sum_{k=0}^{\infty}\Psi_{k}^{(q^{-n})}(b,x|uq^{-1})\Psi_{k}^{(q^{-m})}(c,y|vq^{-1})\frac{z^k}{(q;q)_{k}}\\
        &=u^{-\binom{n}{2}}\g_{n}(bD_{q}|u)\left\{\sum_{k=0}^{\infty}\Psi_{k}^{(q^{-m})}(c,y|vq^{-1})\frac{(xz)^k}{(q;q)_{k}}\right\}\\
        &=u^{-\binom{n}{2}}\g_{n}(bD_{q}|u)\left\{\frac{(1\oplus_{1,v}v^{1-m}cxz)_{q}^{(m)}}{(xyz;q)_{\infty}}\right\}\\
        &=\sum_{k=0}^{n}\qbinom{n}{k}_{q}u^{\binom{k+1}{2}-nk}b^kD_{q}^k\left\{\frac{(1\oplus_{1,v}v^{1-m}cxz)_{q}^{(m)}}{(xyz;q)_{\infty}}\right\}\\
        &=\sum_{k=0}^{n}\qbinom{n}{k}_{q}u^{\binom{k+1}{2}-nk}b^k\sum_{l=0}^{k}\qbinom{k}{l}_{q}q^{l(l-k)}D_{q}^{k}\left\{\frac{1}{(xyz;q)_{\infty}}\right\}D_{q}^{k-l}(1\oplus_{1,v}v^{1-m}q^{l}cxz)_{q}^{(m)}\\
        &=\frac{1}{(xyz;q)_{\infty}}\sum_{k=0}^{n}\qbinom{n}{k}_{q}u^{\binom{k+1}{2}-nk}b^k\\
        &\hspace{1cm}\times\sum_{l=0}^{k}\qbinom{k}{l}_{q}\qbinom{m}{k-l}_{q}v^{\binom{k-l}{2}}(q;q)_{k-l}(1\oplus_{1,v}v^{1-m+k-l}q^{l}cxz)_{q}^{(m-k+l)}(yz)^l(v^{1-m}q^lcz)^{k-l}
    \end{align*}
\end{proof}

\begin{theorem}[{\bf Ordinary Rogers formula}]
\begin{multline*}
    \sum_{n=0}^{\infty}\sum_{m=0}^{\infty}\Psi_{n+m}^{(q^{-k})}(b,x|uq^{-1})y^nz^m\\=\frac{1}{1-xy}\sum_{l=0}^{k}\qbinom{k}{l}_{q}u^{\binom{l+1}{2}-lk}b^l\sum_{i=0}^{l}\qbinom{l}{i}_{q}\frac{y^i}{(qxy;q)_{i}}\frac{z^{l-i}}{(q^ixz;q)_{l-i+1}}.
\end{multline*}
\end{theorem}
\begin{proof}
    \begin{align*}
        &\sum_{n=0}^{\infty}\sum_{m=0}^{\infty}\Psi_{n+m}^{(q^{-k})}(b,x|uq^{-1})y^nz^m\\
        &\hspace{1cm}=u^{-\binom{k}{2}}\sum_{n=0}^{\infty}\sum_{m=0}^{\infty}\g_{k}(bD_{q}|u)\{(xy)^n(xz)^m\}\\
        &\hspace{1cm}=u^{-\binom{k}{2}}\g_{k}(bD_{q}|u)\left\{\sum_{n=0}^{\infty}\sum_{m=0}^{\infty}(xy)^n(xz)^m\right\}\\
        &\hspace{1cm}=u^{-\binom{k}{2}}\g_{k}(bD_{q}|u)\left\{\frac{1}{1-xy}\frac{1}{1-xz}\right\}\\
        &\hspace{1cm}=\sum_{l=0}^{k}\qbinom{k}{l}_{q}u^{\binom{l+1}{2}-lk}b^lD_{q}^{l}\left\{\frac{1}{1-xy}\frac{1}{1-xz}\right\}\\
        &\hspace{1cm}=\sum_{l=0}^{k}\qbinom{k}{l}_{q}u^{\binom{l+1}{2}-lk}b^l\sum_{i=0}^{l}\qbinom{l}{i}_{q}q^{i(i-l)}D_{q}^{i}\left\{\frac{1}{1-xy}\right\}D_{q}^{l-i}\left\{\frac{1}{1-q^ixz}\right\}\\
        &\hspace{1cm}=\sum_{l=0}^{k}\qbinom{k}{l}_{q}u^{\binom{l+1}{2}-lk}b^l\sum_{i=0}^{l}\qbinom{l}{i}_{q}\frac{y^i}{(xy;q)_{i+1}}\frac{z^{l-i}}{(q^ixz;q)_{l-i+1}}\\
        &\hspace{1cm}=\frac{1}{1-xy}\sum_{l=0}^{k}\qbinom{k}{l}_{q}u^{\binom{l+1}{2}-lk}b^l\sum_{i=0}^{l}\qbinom{l}{i}_{q}\frac{y^i}{(qxy;q)_{i}}\frac{z^{l-i}}{(q^ixz;q)_{l-i+1}}
    \end{align*}
\end{proof}

\begin{theorem}[{\bf Rogers formula}]
\begin{multline*}
    \sum_{n=0}^{\infty}\sum_{m=0}^{\infty}\Psi_{n+m}^{(q^{-k})}(b,x|uq^{-1})\frac{y^n}{(q;q)_{n}}\frac{z^m}{(q;q)_{m}}\\
    =\frac{1}{(xy,xz;q)_{\infty}}\sum_{l=0}^{k}\qbinom{k}{l}_{q}u^{\binom{l+1}{2}-kl}b^l\Phi_{l}^{(xz)}(y,z|q).
\end{multline*}
\end{theorem}
\begin{proof}
    \begin{align*}
        &\sum_{n=0}^{\infty}\sum_{m=0}^{\infty}\Psi_{n+m}^{(q^{-k})}(b,x|uq^{-1})\frac{y^n}{(q;q)_{n}}\frac{z^m}{(q;q)_{m}}\\
        &\hspace{1cm}=u^{-\binom{k}{2}}\sum_{n=0}^{\infty}\sum_{m=0}^{\infty}\g_{k}(bD_{q}|u)\left\{\frac{(xy)^n}{(q;q)_{n}}\frac{(xz)^m}{(q;q)_{m}}\right\}\\
        &\hspace{1cm}=u^{-\binom{k}{2}}\g_{k}(bD_{q}|u)\left\{\sum_{n=0}^{\infty}\sum_{m=0}^{\infty}\frac{(xy)^n}{(q;q)_{n}}\frac{(xz)^m}{(q;q)_{m}}\right\}\\
        &\hspace{1cm}=u^{-\binom{k}{2}}\g_{k}(bD_{q}|u)\left\{\frac{1}{(xy;q)_{\infty}}\frac{1}{(xz;q)_{\infty}}\right\}\\
        &\hspace{1cm}=\sum_{l=0}^{k}\qbinom{k}{l}_{q}u^{\binom{l+1}{2}-kl}b^lD_{q}^l\left\{\frac{1}{(xy;q)_{\infty}}\frac{1}{(xz;q)_{\infty}}\right\}\\
        &\hspace{1cm}=\sum_{l=0}^{k}\qbinom{k}{l}_{q}u^{\binom{l+1}{2}-kl}b^l\sum_{i=0}^{l}\qbinom{l}{i}_{q}q^{i(i-l)}D_{q}^i\left\{\frac{1}{(xy;q)_{\infty}}\right\}D_{q}^{l-i}\left\{\frac{1}{(q^ixz;q)_{\infty}}\right\}\\
        &\hspace{1cm}=\sum_{l=0}^{k}\qbinom{k}{l}_{q}u^{\binom{l+1}{2}-kl}b^l\sum_{i=0}^{l}\qbinom{l}{i}_{q}\frac{y^i}{(xy;q)_{\infty}}\frac{z^{l-i}}{(q^ixz;q)_{\infty}}\\
        &\hspace{1cm}=\frac{1}{(xy,xz;q)_{\infty}}\sum_{l=0}^{k}\qbinom{k}{l}_{q}u^{\binom{l+1}{2}-kl}b^l\sum_{i=0}^{l}\qbinom{l}{i}_{q}(xz;q)_{i}y^iz^{l-i}\\
        &\hspace{1cm}=\frac{1}{(xy,xz;q)_{\infty}}\sum_{l=0}^{k}\qbinom{k}{l}_{q}u^{\binom{l+1}{2}-kl}b^l\Phi_{l}^{(xz)}(y,z|q)
    \end{align*}
as claimed.    
\end{proof}


\begin{thebibliography}{99}

\bibitem{orozco1}
R. Orozco, 
Deformed Newton's $(s,t)$-Binomial Series and Generating Functions of Generalized Central Binomial Coefficients and Generalized Catalan Numbers, 
arXiv:2306.07431v3, 
(2024).

\bibitem{orozco2}
R. Orozco, 
Rogers-Szeg\"o Polynomials, $(s,t)$-Derivatives of Partial Theta Function, and Generalized Simplicial $d$-Polytopic Numbers, I
arXiv:2408.08943, 
(2024).

\bibitem{chen1}
W.Y.C. Chen, Z.G. Liu,
Parameter augmenting for basic hypergeometric series, I,
Mathematical Essays in Honor of Gian-Carlo Rota, Eds., B. E. Sagan and R.P. Stanley, Birkh\"auser, Boston, 1998, pp. 111--129.

\bibitem{chen2}
W.Y.C. Chen, Z.G. Liu,
Parameter augmenting for basic hypergeometric series, II,
J. Combin. Theory, Ser. A
\textbf{80} (1997) 175--195.


\bibitem{gasper}
G. Gasper, M. Rahman,
Basic Hypergeometric Series, $2^{nd}$ ed., 
Cambridge University Press, Cambridge, MA, 1990.

\bibitem{cao}
Cao, J., Xu, B. and Arjika, S., 
A note on generalized q-difference equations for general Al-Salam–Carlitz polynomials,
Adv. Differ. Equ., \textbf{2020}, 668, (2020),
https://doi.org/10.1186/s13662-020-03133-7


\end{thebibliography}
\end{document}